\documentclass{amsart}

\usepackage{colonequals}
\usepackage[alphabetic]{amsrefs}
\usepackage{enumitem}
\usepackage{amssymb}

\newtheorem{theorem}{Theorem}[section]
\newtheorem{lemma}[theorem]{Lemma}
\newtheorem{proposition}[theorem]{Proposition}
\newtheorem{corollary}[theorem]{Corollary}

\theoremstyle{definition}

\newtheorem{example}[theorem]{Example}
\newtheorem{question}[theorem]{Question}

\theoremstyle{remark}
\newtheorem{remark}[theorem]{Remark}

\numberwithin{equation}{section}

\newcommand{\bP}{\operatorname{\mathbb{P}}}
\newcommand{\F}{\operatorname{\mathbb{F}}}

\makeatletter
\@namedef{subjclassname@2020}{%
  \textup{2020} Mathematics Subject Classification}
  \makeatother

\begin{document}

\title{Smoothness in pencils of hypersurfaces over finite fields}

\author{Shamil Asgarli}
\address{Department of Mathematics, University of British Columbia, Vancouver, BC V6T 1Z2}
\email{sasgarli@math.ubc.ca}

\author{Dragos Ghioca}
\address{Department of Mathematics, University of British Columbia, Vancouver, BC V6T 1Z2}
\email{dghioca@math.ubc.ca}


\begin{abstract}
We study pencils of hypersurfaces over finite fields $\mathbb{F}_q$ such that each of the $q+1$ members defined over $\mathbb{F}_q$ is smooth. 
\end{abstract}

\maketitle

\section{Introduction}

It is a well-known principle in algebraic geometry that ``most" hypersurfaces in $\mathbb{P}^n$ of a given degree $d$ defined over some field $k$ is smooth. When the underlying field is algebraically closed, we can make this precise by asserting that smooth hypersurfaces of a given degree form a dense open subset of the parameter space under the Zariski topology. When $k = \mathbb{F}_q$ is a finite field, it becomes more subtle to quantify the principle that the number of smooth hypersurfaces of a fixed degree $d$ defined over $\F_q$ is sufficiently large.

As an application of Lang--Weil theorem \cite{LW54}, the proportion of smooth hypersurfaces of degree $d$ defined over $\F_q$ tends to $1$ as $q\mapsto \infty$. While this justifies the abundance of smoothness over finite fields, it does not answer finer questions on the distribution of smooth hypersurfaces. In this paper, we consider the following general question on the existence of smooth hypersurfaces along linear subspaces:

\begin{question}\label{quest:motivation}
Fix a prime power $q$. Let $(n, d, r)$ be any triple of positive integers. Does there exist a linear subspace $\mathcal{L}$ of projective dimension $r$ over $\F_q$ parametrizing degree $d$ hypersurfaces in $\mathbb{P}^n$ such that each of the $\F_q$-members of $\mathcal{L}$ is smooth? 
\end{question}

To be more precise, suppose that we have $r+1$ hypersurfaces $X_0, \dots, X_r$ defined by $X_i = \{F_i=0\}$ for some homogeneous polynomials $F_i\in \F_q[x_0,\dots, x_n]$. We can consider the vector space $\mathcal{L}=\langle F_0, \dots, F_r\rangle$ spanned by $F_i$. If $\mathcal{L}$ has maximal dimension $r+1$, we say that $\mathcal{L}$ has \emph{projective dimension} $r$. By \emph{$\F_q$-members} of $\mathcal{L}$, we mean the elements of $\mathcal{L}$ defined over $\F_q$, or equivalently, the hypersurfaces which are expressible as $X = \{a_0 F_0 + a_1 F_1 + \dots + a_r F_r=0\}$ where $a_i\in \F_q$ for each $0\leq i\leq r$. Note that $\mathcal{L}$ has exactly $\# \bP^r(\F_q) = q^{r}+q^{r-1}+\dots+1$ many $\F_q$-members.

When $r=1$, the space $\mathcal{L}=\langle F_0, F_1\rangle$ is called a \emph{pencil}. The main Question~\ref{quest:motivation} in this special case reduces to the following:

\medskip

\begin{question}~\label{quest:pencils}
Does there exist a pencil $\mathcal{L}$ of hypersurfaces in $\mathbb{P}^n$ of degree $d$ over $\F_q$ such that each of the $q+1$ members of $\mathcal{L}$ defined over $\F_q$ is smooth?
\end{question} 

Our main result asserts that Question~\ref{quest:pencils} has a positive answer when $q$ is sufficiently large compared to $d$. More precisely, we prove the following effective result.

\begin{theorem}\label{thm:pencil}
Let $n, d$ be positive integers with $n\geq 2$ and $d\geq 2$. Suppose that 
$$
q > \left(\frac{1+\sqrt{2}}{2}\right)^2 \left((n+1)(d-1)^{n}\right)^2 \left((n+1)(d-1)^{n}-1\right)^2\left((n+1)(d-1)^{n}-2\right)^2.
$$
Then there exists a pencil of hypersurfaces of degree $d$ in $\bP^n$ such that each of the $q+1$ members defined over $\F_q$ is smooth.
\end{theorem}

We remark that the case $n=2$ and $d=2$ in Question~\ref{quest:pencils} was investigated in our previous work \cite{AG22}*{Example 2.6}. In that example, the smooth conics $D_1$, $D_2$, \dots, $D_{q+1}$ arise from a pencil of conics for a suitably constructed base locus $B$, which consists of $3$ points Galois-conjugated over $\F_{q^3}$, and a single $\F_q$-point. The construction works for any prime power $q$, and it follows that Question~\ref{quest:pencils} has a positive answer in the case $(n, d)=(2,2)$ over $\F_q$ for each $q$.

\begin{remark}
It is impossible for Question~\ref{quest:motivation} to have a positive answer for all possible choices of $(n, d, r)$. Indeed, $r$ must be strictly less than the dimension of the projective space parametrizing all degree $d$ hypersurfaces in $\bP^n$. In particular, it is necessary that $r \leq \binom{n+d}{d}-2$. However, this condition is not sufficient in general; indeed, Example~\ref{ex:(n,d,r)=(2,2,3)} below shows that Question~\ref{quest:motivation} has a negative answer in the case $(n, d, r)=(2, 2, 3)$ over $\F_2$, and yet $r = 3 < \binom{2+2}{2}-2=4$.  Nevertheless, we expect that Question~\ref{quest:pencils} always has a positive answer. In other words, we believe that Theorem~\ref{thm:pencil} should be true with no additional hypothesis on $q$ and $d$. \end{remark}

\begin{example}\label{ex:(n,d,r)=(2,2,2)} Let $q=2$. Consider the polynomials $f_0=x^2+y^2+xz$, $f_1=xy+xz+z^2$ and $f_2=x^2+yz$ in $\F_2[x,y,z]$. One can check that the $7$ conics defined by $f_0, f_1, f_2, f_0+f_1, f_0+f_2, f_1+f_2, f_0+f_1+f_2$ are all smooth. In other words, the conic defined by
$$
a_0 f_0 + a_1 f_1 + a_2 f_2 = 0
$$
is smooth for each $[a_0:a_1:a_2]\in \bP^2(\F_2)$. Thus, Question~\ref{quest:motivation} has a positive answer for the case $(n, d, r)=(2, 2, 2)$ over $\F_2$. \end{example}

\begin{example}\label{ex:(n,d,r)=(2,2,3)}
Let $q=2$. When $(n, d, r)=(2,2,3)$, we are searching for $r+1=4$ conics $\{f_0=0\}, \{f_1=0\}, \{f_2=0\}, \{f_3=0\}$ such that $\{a_0 f_0 + a_1 f_1 + a_2 f_2 + a_3 f_3=0\}$ is a smooth conic for each of the $15$ values $[a_0 : a_1: a_2 : a_3]\in \bP^3(\F_2)$. There are $q^5-q^2=28$ smooth conics over $\F_2$. Checking all possible $\binom{28}{4}=20475$ subsets of size $4$ in Macaulay2, we see that no such $4$-tuple $(f_0, f_1, f_2, f_3)$ exists. Thus, Question~\ref{quest:motivation} has a negative answer for the case $(n, d, r)=(2, 2, 3)$ over $\F_2$.
\end{example}

Based on the examples above, it would be interesting to characterize all triples $(n, d, r)$ for a given finite field $\F_q$ such that Question~\ref{quest:motivation} has a positive answer.

The present paper is organized as follows. In Section~\ref{sect:main} we give a proof of our main theorem. In Section~\ref{sect:quadrics} we give a concrete approach to settle Question~\ref{quest:pencils} affirmatively in the special case $(n, d)=(3,2)$ for all $q$. 

\medskip

\textbf{Acknowledgements.} We thank Lian Duan and Felipe Voloch for discussions regarding the problem investigated in this paper. The first author is supported by a Postdoctoral Research Fellowship and an NSERC PDF award at the University of British Columbia. The second author is supported by an NSERC Discovery grant.

\section{Main result}~\label{sect:main}

We begin the section by explaining the strategy behind the proof of Theorem~\ref{thm:pencil}. The key observation is that a pencil of hypersurfaces of degree $d$ can be viewed as a line inside the parameter space $\bP^{N}$ with $N=\binom{n+d}{d}-1$ whose points correspond to degree $d$ hypersurfaces in $\bP^n$. It turns out that the singular hypersurfaces of degree $d$ are parametrized by a hypersurface $\mathcal{D}$ inside $\bP^N$, known as the \emph{discriminant}. The singular members of a given pencil $\mathcal{L}\cong \bP^1$ precisely corresponds to the intersection $\mathcal{L}\cap \mathcal{D}$. Thus, it suffices to find an $\F_q$-line that meets the discriminant hypersurface only at non-$\F_q$-points. We prove a general result that guarantees such a line for any hypersurface $X$ in Proposition~\ref{prop:hypersurfaces}. The proof of this latter result naturally reduces (after slicing $X$ by a suitable plane defined over $\F_q$) to the case when $X$ is a plane curve. The case of plane curves is proved separately in Proposition~\ref{prop:curves} and contains the novel part of the paper.   
 
Before we proceed with more technical details, we clarify the usage of the word ``irreducible". Given a hypersurface $X\subset \bP^n$ defined by $\{F=0\}$ over $\F_q$, we say that $X$ is \emph{irreducible} if $F$ cannot be factored into a product of two polynomials of smaller degree in $\F_q[x_0, \dots, x_n]$. We say that $X$ is \emph{geometrically irreducible} if $X$ is irreducible when viewed over the algebraic closure $\overline{\F}_q$.

As alluded above, we begin by proving an effective result that guarantees the existence of an $\F_q$-line whose intersection with a given plane curve has no $\F_q$-points.

\begin{proposition}\label{prop:curves}
Suppose that $C\subset \bP^2$ is a geometrically irreducible curve of degree $\delta \geq 2$ defined over $\F_q$. If $q>\left(\frac{1+\sqrt{2}}{2}\right)^2 \delta^2 (\delta-1)^2 (\delta-2)^2$, then we can find an $\F_q$-line $L\subset\bP^2$ such that the intersection $C\cap L$ has no $\F_q$-points.
\end{proposition} 

\begin{proof}
Let $(\bP^2)^{\ast}(\F_q)$ denote the set of all $\F_q$-lines in $\bP^2$.  Consider the following finite set:
$$
\mathcal{I} = \{ (L, P) \ | \ L \in (\mathbb{P}^2)^{\ast} (\F_q) \text{ and } P\in (L\cap C)(\F_q) \} 
$$
where $(L\cap C)(\F_q)$ stands for the set of $\F_q$-points of the intersection $L\cap C$. We count the cardinality of $\mathcal{I}$ in two different ways. First, fixing a point $P\in C(\F_q)$, there are exactly $q+1$ distinct $\F_q$-lines $L$ passing through $P$, which yields
\begin{equation}\label{eq:cardinality-I-first}
    \# \mathcal{I} = (q+1)\cdot N_q(C)
\end{equation}
where $N_q(C)$ denotes $\# C(\F_q)$. On the other hand, we can fix an $\F_q$-line $L$, and let $m_{L}$ denote $\# (C\cap L)(\F_q)$. Then
\begin{equation}\label{eq:cardinality-I-second}
    \# \mathcal{I} = \sum_{L\in (\bP^2)^{\ast}(\F_q)} m_{L}
\end{equation}
For each $0\leq i\leq \delta$, we define
$$
\mathcal{T}_{i} \colonequals \{ L \in (\bP^{2})^{\ast}(\F_q) \ | \ m_L = i\}.
$$
It is clear that $(\bP^2)^{\ast}(\F_q)$ is a disjoint union of $\mathcal{T}_i$ for $0\leq i\leq \delta$. Combining \eqref{eq:cardinality-I-first} and \eqref{eq:cardinality-I-second}, we obtain
\begin{equation}\label{eq:I-relation}
    (q+1)\cdot N_q(C) =  \sum_{L\in (\bP^2)^{\ast}(\F_q)} m_{L} =  \sum_{L\in \mathcal{T}_1} m_{L} + \sum_{L\in \mathcal{T}_2} m_{L} + \cdots + \sum_{L\in \mathcal{T}_{\delta}} m_{L}.
\end{equation}

Next, consider the following finite set,
$$
\mathcal{J} = \{ (L, \{P_1, P_2\}) \ | \ L\in (\bP^2)^{\ast}(\F_q) \text{ and } \{P_1, P_2\}\subset (L\cap C)(\F_q) \}.
$$
Note that the notation $\{P_1, P_2\}$ implicitly assumes $P_1\neq P_2$. Since $P_1$ and $P_2$ uniquely determine $L$, we have $\# \mathcal{J}=\binom{N_q(C)}{2}$. On the other hand,
$$
\# \mathcal{J} = \sum_{L\in\mathcal{T}_2} \binom{2}{2} + \sum_{L\in\mathcal{T}_3} \binom{3}{2} + \cdots + \sum_{L\in\mathcal{T}_{\delta}} \binom{\delta}{2}
$$
because a given element in $\mathcal{T}_i$ contributes exactly $\binom{i}{2}$ pairs to its second coordinate. 
Let $t_i = \# \mathcal{T}_i$. Combining the two formulas for the cardinality of $\mathcal{J}$, we obtain:
\begin{equation}\label{eq:constraint}
\binom{N_q(C)}{2} = \sum_{i=2}^{\delta} t_i \cdot \binom{i}{2}
\end{equation}
Consequently,
$$
 N_q(C)\cdot (N_q(C)-1) = \sum_{i=2}^{\delta} t_i \cdot i(i-1) \leq \delta \sum_{i=2}^{\delta} t_i \cdot (i-1)
$$
we obtain
\begin{equation}\label{ineq:minimum-sum}
  \sum_{i=2}^{\delta} t_i \cdot (i-1) \geq \frac{N_q(C)\cdot (N_q(C)-1)}{\delta}
\end{equation}

After rewriting \eqref{eq:I-relation} as,
\begin{equation*}
(q+1)\cdot N_q(C) = \sum_{i=1}^{\delta} t_i \cdot i = \sum_{i=1}^{\delta} t_i + \sum_{i=2}^{\delta} t_i (i-1) 
\end{equation*} 
and substituting \eqref{ineq:minimum-sum}, we obtain
$$
(q+1)\cdot N_q(C) \geq \sum_{i=1}^{\delta} t_i + \frac{N_q(C)\cdot (N_q(C)-1)}{\delta}
$$
Thus,
$$
\sum_{i=1}^{\delta} t_i \leq (q+1)\cdot N_q(C) - \frac{N_q(C)\cdot (N_q(C)-1)}{\delta}
$$
Using the fact that $\sum_{i=0}^{\delta} t_i = q^2+q+1$, we obtain:
\begin{equation}\label{ineq:t_0}
t_0 = (q^2+q+1)-\sum_{i=1}^{\delta} t_i \geq (q^2+q+1) - (q+1)\cdot N_q(C) + \frac{N_q(C)\cdot (N_q(C)-1)}{\delta}
\end{equation}
Our goal is to show that $t_0>0$ because $t_0$ exactly counts the $\F_q$-lines where $(L\cap C)(\F_q)=\emptyset$. We will use the Hasse-Weil inequality for geometrically irreducible plane curves \cite{AP96}*{Corollary 2.5}, which states that
\begin{equation}\label{ineq:Hasse-Weil}
q+1 - (\delta-1)(\delta-2)\sqrt{q} \leq N_q(C)\leq q+1+(\delta-1)(\delta-2)\sqrt{q}
\end{equation}
We obtain,
\begin{align*}
    t_0 &\geq (q^2+q+1) - (q+1)\cdot N_q(C) + \frac{N_q(C)\cdot (N_q(C)-1)}{\delta} \\
    & = (q^2+q+1) - N_q(C)\left(\frac{\delta(q+1) - N_q(C) +1}{\delta} \right) \\
    & \geq (q^2+q+1) - N_q(C)\left(\frac{\delta(q+1) - (q+1-(\delta-1)(\delta-2)\sqrt{q}) +1}{\delta} \right) \\
    &= (q^2+q+1) - N_q(C)\left(\frac{(\delta-1)(q+1) +(\delta-1)(\delta-2)\sqrt{q} +1}{\delta} \right)
\end{align*}
In order to prove that $t_0 > 0$, we will focus on proving 
\begin{align*}
    \delta(q^2+q+1) > N_q(C)\left( (\delta-1)(q+1)+(\delta-1)(\delta-2)\sqrt{q}+1\right)
\end{align*}
Using \eqref{ineq:Hasse-Weil}, it suffices to show that,
\begin{align*}
      \delta(q^2+q+1) > (q+1+(\delta-1)(\delta-2)\sqrt{q})\left( (\delta-1)(q+1)+(\delta-1)(\delta-2)\sqrt{q}+1\right)
\end{align*}
After simplifying and rearranging the terms, it is enough to prove that,
\begin{align*}
      q^2 & > \delta(\delta-1)(\delta-2) q\sqrt{q} + [(\delta-1)^2(\delta-2)^2 +(\delta-1)] q \\
      & \ \ \ + (\delta+1)(\delta-1)(\delta-2)\sqrt{q}.
\end{align*}
This is clearly true when $\delta=2$, so we will assume $\delta\geq 3$ for the rest of the proof. Assuming $q>\frac{1}{2}(\delta+1)(\delta-1)(\delta-2)\sqrt{q}$, it suffices to prove that,
\begin{align*}
  q > \delta(\delta-1)(\delta-2)\sqrt{q} + [(\delta-1)^2(\delta-2)^2 + (\delta+1)] 
\end{align*}
Using the quadratic formula and $\delta^2(\delta-1)^2(\delta-2)^2 \geq 4(\delta-1)^2(\delta-2)^2+4(\delta+1)$ for $\delta\geq 3$, one can check that the desired inequality holds provided that, 
$$
q > \left(\frac{1+\sqrt{2}}{2}\right)^2 \delta^2 (\delta-1)^2 (\delta-2)^2
$$
Finally, we have to make sure that our earlier assumption $q>\frac{1}{2}(\delta+1)(\delta-1)(\delta-2)\sqrt{q}$ is valid. This is indeed the case since
$$
q > \left(\frac{1+\sqrt{2}}{2}\right)^2 \delta^2 (\delta-1)^2 (\delta-2)^2 > \left(\frac{1}{2} (\delta+1)(\delta-1)(\delta-2)\right)^2
$$
holds for $\delta\geq 3$. This completes the proof. \end{proof}

\begin{remark} The conclusion of Proposition~\ref{prop:curves} continues to hold when $X$ is irreducible over $\F_q$. Indeed, if $X$ is irreducible but not geometrically irreducible, we can write $X=X_1\cup X_2\cup\cdots \cup X_s$ where each $X_i$ is geometrically irreducible and $s\geq 2$. Without loss of generality, we have $X_i = \phi^{i}(X_1)$ where $\phi\colon \bP^2\to\bP^2$ denotes the Frobenius map $[x:y:z]\mapsto [x^q:y^q:z^q]$. As a result, each $\F_q$-point of $X$ belongs to $X_i$ for all $1\leq i\leq s$. In particular, each $\F_q$-point of $X$ must belong to the intersection $X_1\cap X_2$. Let $\delta=\deg(X)$, $\delta_1=\deg(X_1)$, and $\delta_2=\deg(X_2)$. Applying B\'ezout's theorem, we obtain that $\# X(\F_q)\leq \delta_1 \delta_2 \leq \frac{(\delta_1+\delta_2)^2}{4} \leq \frac{\delta^2}{4}$. Thus, the total number of $\F_q$-lines passing through some $\F_q$-point of $X$ is at most $\frac{\delta^2}{4}\cdot(q+1) < q^2+q+1$ because $q>\left(\frac{1+\sqrt{2}}{2}\right)^2\delta^2 (\delta-1)^2 (\delta-2)^2 > \frac{\delta^2}{4}$ for $\delta\geq 3$. For $\delta=2$, the conclusion is still true because $\frac{\delta^2}{4}=1$ and $q\geq 2$. In particular, there exists an $\F_q$-line $L$ which does not pass through any $\F_q$-point of $X$, as desired. \end{remark}

\begin{remark}
We remark that an alternative way to prove Proposition~\ref{prop:curves} is to use the Chebotarev Density Theorem for varieties over finite fields \cite{Ent21}*{Theorem 3}. From a given plane curve $C$, one can construct a finite \'etale map $f\colon X\to Y$ where $Y=(\bP^2)^{\ast}$ parametrizes lines in $\bP^2$ and $f^{-1}(L)$ records the intersection $L\cap C$. Thus, the problem reduces to finding an $\F_q$-point $L\in Y$ such that $f^{-1}(L)$ has no fixed point in its orbit under the Frobenius action. Since the orbit decomposition should behave uniformly as $q\to\infty$, we obtain the desired conclusion for $q\gg d$. Making the bound $q\gg d$ effective via this method is much more complex as there are many implicit constants. We believe that our approach has advantages of being both elementary and yet providing an explicit bound $q>\left(\left(\frac{1+\sqrt{2}}{2}\right) \delta(\delta-1)(\delta-2)\right)^2$. 
\end{remark}
 
Using a slicing method, we can generalize the previous result to any geometrically irreducible hypersurface of degree at least $4$. 

\begin{proposition}\label{prop:hypersurfaces}
Suppose $X\subset \bP^n$ is a geometrically irreducible hypersurface of degree $\delta\geq 4$ defined over $\F_q$. If $q>\left(\frac{1+\sqrt{2}}{2}\right)^2 \delta^2 (\delta-1)^2 (\delta-2)^2$, we can find an $\F_q$-line $L\subset \bP^n$ such that the intersection $X\cap L$ has no $\F_q$-points.
\end{proposition}

\begin{proof}
Using Kaltofen's result \cite{Kal91}*{Theorem 5}, which was made explicit in \cite{CM06}*{Corollary 3.2}, we can find an $\F_q$-plane $H\subset \bP^n$ such that $C\colonequals X\cap H$ is a geometrically irreducible plane curve provided that $q >\frac{1}{2}\left(3\delta^4-4\delta ^3+5\delta^2\right)$. We can use this result since it is straightforward to verify that
$$
\left(\frac{1+\sqrt{2}}{2}\right)^2 \delta^2 (\delta-1)^2 (\delta-2)^2 \geq \frac{1}{2}\left(3\delta^4-4\delta ^3+5\delta^2\right)
$$
holds for all $\delta\geq 4$. Applying Proposition~\ref{prop:curves} to the curve $C$ inside $H\cong \bP^2$, we immediately obtain the desired result. 
\end{proof}

We are now ready to present the proof of the main result. 

\begin{proof}[Proof of Theorem~\ref{thm:pencil}]
We apply Proposition~\ref{prop:hypersurfaces} to the discriminant hypersurface $X=\mathcal{D}_{n, d}$ which parameterizes all singular degree $d$ hypersurfaces in $\bP^n$. It is known that $X$ is geometrically irreducible \cite{EH16}*{Proposition 7.1} and the degree of $X$ is $\delta = (n+1)(d-1)^n$ by \cite{EH16}*{Proposition 7.4}. Note that the inequality $\delta \geq 4$ always holds except for the case $(n, d)=(2, 2)$ which was already handled in our previous paper \cite{AG22}*{Example 2.6}. We obtain an $\F_q$-line $L$ whose intersection with $X$ has no $\F_q$-points; this line $L\cong \bP^1$ corresponds to a pencil of hypersurfaces of degree $d$ such that each of the $q+1$ distinct $\F_q$-members is smooth.
\end{proof}

\section{The pencil of quadric surfaces}~\label{sect:quadrics}

In this section, we show that Question~\ref{quest:pencils} has a positive answer when $(n, d)=(3, 2)$ for \emph{all} $q$. There are at least two approaches to show that there exists a pencil of quadric surfaces in $\bP^3$ where each $\F_q$-member is smooth. First, applying Theorem~\ref{thm:pencil} directly with $(n, d)=(3, 2)$, there exists a desired pencil provided that 
$$
q > \left(\frac{1+\sqrt{2}}{2}\right)^2 \cdot 4^2\cdot 3^2 \cdot 2^2 \approx 839.3
$$
Thus, we only need to check the conclusion for all prime powers $q\leq 839$. This can be achieved by randomly sampling a pair of quadrics using a computer algebra system and searching until one finds a pencil that works. The second method, presented below, is more conceptual and directly constructs the desired pencil.

We will first focus on the case when $q$ is odd, and afterwards consider the case when $q$ is even. We begin with an elementary lemma which will help us in showing the irreducibility of a certain quadratic in the construction later.

\begin{lemma}\label{lemma:non-square}
Suppose $\F_q$ is a finite field with $q$ odd. Then there exists a square $s\in \F_q$ such that $s+1$ is a non-square.
\end{lemma}

\begin{proof}
Assume, to the contrary, that for each square $s\in \F_q$, the element $s+1$ is also a square. Write $q=p^r$ where $p$ is an odd prime number and $r\geq 1$. Observe that for each $x\in \F_q$, either all of the numbers $x,x+1, \dots, x+p-1$ are squares, or none of them is a square. Indeed, once there exists a square in this sequence, our assumption leads to all of them being squares.

Now, there are exactly
$$
1 + \frac{q-1}{2} = 1 + \frac{p^r-1}{2} = \frac{p^r+1}{2}
$$
squares in $\F_q$. However, the observation above implies that the number of squares must be a multiple of $p$. This is a contradiction, as $p$ does not divide $\frac{p^r+1}{2}$. \end{proof}

\begin{corollary}\label{cor:non-square}
Suppose $\F_q$ is a finite field with $q$ odd. Then there exists an element $c\in \F_q$ such that $c^2-2c+5$ is a non-square in $\F_q$.
\end{corollary}

\begin{proof}
By Lemma~\ref{lemma:non-square}, there exists a square $s=b^2\in \F_q$ such that $s+1\in\F_q$ is a non-square. Let $c=2b+1$. Then $c^2-2c+5=4(b^2+1)=4(s+1)$ is a non-square. \end{proof}

The following construction shows that Question~\ref{quest:pencils} has a positive answer in the case $(n, d)=(3, 2)$ over $\F_q$ for each odd $q$. 

\begin{proposition}\label{prop:quadric-odd-q}
Let $q$ be an odd prime power. Let $c\in \F_q$ be such that $c^2-2c+5$ is a non-square. Consider the homogeneous quadratic polynomials
\begin{align*}
    f_0 &= x^2 + y^2 + z^2 + w^2 \\
    f_1 &= xy + yz + zw + c wx
\end{align*}
in $\F_q[x,y,z,w]$. Then each of the $q+1$ distinct $\F_q$-members of the pencil $\langle f_0, f_1 \rangle$ is a smooth quadric surface in $\bP^3$.
\end{proposition}

\begin{proof}
Recall that an arbitrary element of the pencil is defined by a polynomial $h=s f_0 + t f_1$ where $[s:t]\in \bP^1$. We want to show that none of the singular members of the pencil is defined over $\F_q$. Given $h = s f_0 + t f_1$, we have:
\begin{align*}
    h_x &= 2sx+t(y+cw), \ \ \ \ \ \ \ h_y = 2sy+t(x+z),  \\ 
    h_z &= 2sz+t(y+w), \ \ \ \ \ \ \ \ h_w = 2sw+t(z+cx).
\end{align*}
The singular points of $\{h=0\}$ must satisfy $h_x=h_y=h_z=h_w=0$. We express these linear equations in matrix notation:
\begin{align*}
    \begin{pmatrix}
    2s & t & 0 & ct \\
    t & 2s & t & 0 \\
    0 & t & 2s & t \\
    ct & 0 & t & 2s
    \end{pmatrix} \cdot \begin{pmatrix} 
    x \\ y \\ z \\ w 
    \end{pmatrix} = \begin{pmatrix}
    0 \\ 0 \\ 0 \\ 0
    \end{pmatrix}
\end{align*}
In particular, the determinant of the matrix above must vanish. One can check that the determinant is equal to:
$$
\left((1-c)t^{2} + 2(c+1)st-4s^{2}\right)\left((1-c)t^{2}-2(c+1)st-4s^{2}\right)
$$
We claim that each of the quadratic factors is irreducible over $\F_q$. Indeed, the discriminant of both quadratics is 
$$
4(c+1)^2+16(1-c)=4c^2-8c+20 = 4(c^2-2c+5)
$$
which is a non-square in $\F_q$ by our choice of $c$. Thus, the four roots of the determinant above are not in $\F_q$, which implies that all the $\F_q$-members of the pencil are smooth. \end{proof}

Next, we consider the case when $q$ is even. We begin with a quick lemma on the irreducibility of quadratic polynomials in characteristic $2$.

\begin{lemma}\label{lemma:irreducible-even-q}
Let $q$ be an even prime power. Then there exists an element $c\in \F_q$ such that $t^2+t+c$ is an irreducible polynomial in $\F_q[t]$.
\end{lemma}

\begin{proof}
Let $K$ be a field extension of $\F_q$ with $[K:\F_q]=2$. Let $u\in K\setminus\F_q$. The minimal polynomial of $u$ over $\F_q$ is given by $\phi(t)=t^2+bx+d\in \F_q[t]$ for some $b, d\in \F_q$. Note that $b\neq 0$ because $\phi$ is irreducible and $\operatorname{char}(\F_q)=2$. Observe that
$$
\psi(t)\colonequals \frac{1}{b^2} \phi(t) = \left(\frac{t}{b}\right)^2+\left(\frac{t}{b}\right) + \frac{d}{b^2}
$$
is an irreducible polynomial too. Letting $c=d/b^2 \in \F_q$, we see that $\psi(tb)=t^2+t+c$ is also an irreducible polynomial in $\F_q[t]$. \end{proof}

The following construction shows that Question~\ref{quest:pencils} has a positive answer in the case $(n, d)=(3, 2)$ over $\F_q$ for each even $q$. 

\begin{proposition}\label{prop:quadric-even-q}
Let $q$ be an even prime power. Let $c\in \F_q$ be such that $t^2+t+c$ is an irreducible polynomial in $\F_q[t]$. Consider the homogeneous quadratic polynomials
\begin{align*}
    f_0 &= x^2 + y^2 + xy + yz + czw \\
    f_1 &= x^2 + z^2 + yz + xw
\end{align*}
in $\F_q[x,y,z,w]$. Then each of the $q+1$ distinct $\F_q$-members of the pencil $\langle f_0, f_1 \rangle$ is a smooth quadric surface in $\bP^3$.
\end{proposition}

\begin{proof}
Suppose $h=sf_0+tf_1$ defines a singular element of the pencil where $[s:t]\in\bP^1$. We have,
\begin{align*}
    h_x &= sy + tw,  \ \ \ \ \ \ \ \ \ \ \ \ \ \ \ \ \ \ h_y = s(x+z)+tz, \\
    h_z &= s(y+cw)+ty,  \ \ \ \ \ \ \ \ \ h_w = csz +tx.
\end{align*}
The singular points of $\{h=0\}$ must satisfy $h_x=h_y=h_z=h_w=0$. These linear conditions on $s$ and $t$ can be expressed in matrix notation:
\begin{align*}
\begin{pmatrix}
0 & s & 0 & t \\
s & 0 & s+t & 0 \\
0 & s+t & 0 & cs \\
t & 0 & cs & 0 
\end{pmatrix} \cdot \begin{pmatrix}  x \\ y\\ z\\ w 
\end{pmatrix} = \begin{pmatrix}
0 \\ 
0 \\
0 \\
0
\end{pmatrix}
\end{align*}
Consequently, the determinant of the matrix on the left hand side must vanish. One can check that the determinant is equal to $(t^2+st + c s^2)^2$ in characteristic $2$. Since $t^2+t+c$ is an irreducible polynomial in $\F_q[t]$ by our choice of $c\in\F_q$, the binary form $t^2 + st + cs^2$ does not vanish for each $[s:t]\in \bP^{1}(\F_q)$. In particular, each $\F_q$-member of the pencil is smooth.
\end{proof}

\end{document}